\documentclass{amsart}

\usepackage{amscd}
\usepackage{amssymb}
\usepackage{mathrsfs}
\usepackage[all]{xy}

\usepackage{amsthm}

\newtheorem{thm}[equation]{Theorem}
\newtheorem{lem}[equation]{Lemma}
\newtheorem{cor}[equation]{Corollary}
\newtheorem{prop}[equation]{Proposition}

\newtheorem*{thm*}{Theorem}
\newtheorem*{prop*}{Proposition}
\newtheorem*{cor*}{Corollary}
\newtheorem*{lem*}{Lemma}
\newtheorem*{MT*}{Main Theorem}

\newtheorem*{ques*}{Question}

\theoremstyle{definition} %
\newtheorem{defn}[equation]{Definition}
\newtheorem*{defn*}{Definition}

\newtheorem{eg}[equation]{Example}

\theoremstyle{remark} %
\newtheorem{rmk}[equation]{Remark}

\newtheorem*{rmk*}{Remark}
\newtheorem*{rmks*}{Remarks}

\newcommand{\ot}{\otimes}

\renewcommand{\P}{\mathbb{P}}

\newcommand{\darkrad}{0.115}

\DeclareRobustCommand{\upstrut}[1]{\rule{0mm}{#1 ex}}

\DeclareMathOperator{\Lie}{Lie}

\newcommand{\C}{{\mathbb{C}}}        
\newcommand{\Q}{{\mathbb{Q}}}        
\newcommand{\R}{{\mathbb{R}}}        
\newcommand{\Z}{{\mathbb{Z}}}        

\newcommand{\Zm}[1]{\Z/{#1}\Z}

\newcommand{\la}{\lambda}
\newcommand{\D}{\Delta}

\newcommand{\aff}{{\mathbb{A}}}      

\newcommand{\oddots}{{\mathinner{\mkern1mu\raise1pt\vbox{\kern7pt\hbox{.}}\mkern2mu\raise4pt\hbox{.}\mkern2mu\raise7pt\hbox{.}\mkern1mu}}}

\newcommand{\s}{\sigma}

\newcommand{\ksep}{k_{{\mathrm{sep}}}}

\newcommand{\kx}{k^\times}

\newcommand{\qform}[1]{{\left\langle{#1}\right\rangle}}                   

\DeclareMathOperator{\Spin}{Spin}           
\newcommand{\Sp}{\mathrm{Sp}}
\DeclareMathOperator{\SL}{SL}
\DeclareMathOperator{\GL}{GL}

\DeclareMathOperator{\HSpin}{HSpin}

\newcommand{\SO}{\mathrm{SO}}

\DeclareMathOperator{\Ad}{Ad}

\newcommand{\Gm}{\mathbb{G}_m}

\DeclareMathOperator{\Stab}{Stab}

\DeclareMathOperator{\Spec}{Spec}

\DeclareMathOperator{\Gal}{Gal}

\DeclareMathOperator{\im}{im}

\DeclareMathOperator{\Int}{Int}
\DeclareMathOperator{\chr}{char}
\DeclareMathOperator{\car}{char}

\DeclareMathOperator{\Id}{Id}

\DeclareMathOperator{\End}{End}

\DeclareMathOperator{\aut}{Aut}
\DeclareMathOperator{\Aut}{Aut}

\newcommand{\Hom}{{\mathrm{Hom}}}

\newcommand{\ra}{\rightarrow}

\newcommand{\E}{\mathrm{E}}
\DeclareMathOperator{\Pf}{Pf}
\DeclareMathOperator{\GSp}{GSp}
\renewcommand{\sc}{{\mathrm{sc}}}
\renewcommand{\O}{\mathscr{O}}

\renewcommand{\k}{K}
\renewcommand{\kx}{\k^\times}
\newcommand{\kalg}{\k_{{\mathrm{alg}}}}
\renewcommand{\ksep}{\k_{{\mathrm{sep}}}}

\newcommand{\Xb}{\bar{X}}

\newcommand{\vp}{v^+}

\newcommand{\g}{\gamma}

\DeclareMathOperator{\Skew}{Skew}
\DeclareMathOperator{\Symm}{Symm}
\newcommand{\Sym}{\Symm}

\DeclareMathOperator{\adj}{adj}

\DeclareMathOperator{\GO}{GO}

\newcommand{\Gb}{\overline{G}}

\newcommand{\Gt}{\widetilde{G}}
\newcommand{\Lt}{\widetilde{L}}
\newcommand{\rt}{\widetilde{\rho}}
\newcommand{\Zt}{\widetilde{Z}}

\numberwithin{equation}{section}

\begin{document}

\title[Linear preservers and representations]{Linear preservers and representations with a 1-dimensional ring of invariants}

\begin{abstract}
We determine the group of  linear transformations on a vector space $V$ that preserve a polynomial function $f$ on $V$ for several interesting pairs $(V,f)$, using the theory of semisimple algebraic groups.
\end{abstract}


\author{H. Bermudez}
\author{S. Garibaldi}
\author{V. Larsen}
\address{Department of Mathematics and Computer Science, MSC W401, 400 Dowman Dr., Emory University, Atlanta, GA 30322, USA}
\thanks{The research for this article was partially supported by  NSA grant no.~H98230-11-1-0178.}


\maketitle

In an 1897 paper \cite{Frobenius}, Frobenius proved that every linear transformation of the $n$-by-$n$ real
matrices that preserves the determinant is of the form
\[
X \mapsto AXB \quad \text{or} \quad X \mapsto AX^t B
\]
for some $A, B \in \GL_n(\R)$ such that $\det(AB) = 1$.  This is the basic example of a  \emph{linear preserver problem} (LPP): one is given a finite-dimensional vector space $V$ over a field $K$ and a polynomial function $f \!: V \ra K$ and one wants to determine the linear transformations of $V$ that preserve $f$.  Since Frobenius, many  such problems have been solved, see for example the surveys \cite{LiPierce:AMM}, \cite{Pierce:survey},  \cite{LiTsing}, and \cite{Marcus:AMMsurv}.  We develop here a general method that solves several new problems, see  Examples \ref{spinors} and \ref{nilpotents} and Corollaries \ref{e6.adj}, \ref{cubics}, \ref{Sp6}, \ref{SL6},  \ref{HSpin12}, and \ref{blackholes}.

Our method is to introduce an auxiliary group $G \subset \GL(V)$ that is semisimple and such that $V$ is an irreducible representation or Weyl module of $G$.  In section \ref{norm.sec}, we determine the normalizer $N_{\GL(V)}(G)$ of $G$ in $\GL(V)$.  We prove in Theorem \ref{rk1} that this subgroup equals the stabilizer of a closed $G$-orbit $\O$ in the projective space $\P(V)$, i.e., $\Stab_{\GL(V)}(\O) = N_{\GL(V)}(G)$.  Using this result, in sections \ref{RRS.sec} and \ref{FTS.sec}, we solve two families of LPPs by reducing the problems in each family to determining this stabilizer.  These two families consist of representations $V$ with a 1-dimensional ring of $G$-invariant functions generated by $f$ and are examples of prehomogenous vector spaces of parabolic type; the two families correspond to the cases where the unipotent radical $U$ of the parabolic subgroup is abelian (i.e., $[U,U] = 0$) or $[U,U]$ is 1-dimensional respectively, and we use the general results on representations in these families from \cite{RRS}, \cite{Roe:extra}, and \cite{Helenius}.

Besides obtaining new results, we also recover many known solutions to linear preserver problems.  The generality of our method is in contrast to many of the proofs in the literature, which typically are highly dependent on the particular choice of $V$ and $f$.  (The arguments in \cite{PlDj} and \cite{Guralnick:LPP2} are notable exceptions.)  Further, we require only very weak assumptions about the field $K$ (at most we require that the characteristic is $\ne 2, 3$) and determine the preserver precisely (and not just its identity component or Lie algebra).

\subsection*{Applications of solutions to LPPs} Linear preserver problems arise naturally in algebra, sometimes in non-obvious ways.  For example, every associative division algebra $D$ that is finite-dimensional over its center $\k$ has a ``generic characteristic polynomial" generalizing the notion of characteristic polynomial on $n$-by-$n$ matrices.  Its coefficients are polynomial functions $E_r \!: D \ra \k$ for $1 \le r \le \sqrt{\dim_\k D}$ where $E_r$ has degree $r$.  By determining the preserver of $E_r$, Waterhouse proved that $D$ is determined up to isomorphism or anti-isomorphism by $E_r$ for any $r \ge 3$, see \cite{Wa:linpres} or \cite[Cor.~4]{Wa:basic}.  This in turn gives a result on the essential dimension of central simple algebras, see \cite{Florence:higher}.

\subsection*{Notation}
An \emph{affine group scheme} $G$ over a field $\k$ is a representable functor from commutative $\k$-algebras to the category of groups, i.e., it is given by $S \mapsto \Hom_{\text{$\k$-alg.}}(\k[G], S)$ for some $\k$-algebra $\k[G]$ and every $\k$-algebra $S$, see \cite{Wa}.  We are mostly concerned with the case where $G$ is a (linear) \emph{algebraic group}, i.e., where $\k[G]$ is smooth of finite type over $\k$.  The ``concrete" group $G(S)$ is  called the \emph{group of $S$-points}.  Below, we maintain the distinction between $G$ and $G(S)$ except for the group $\GL(V)$ of linear transformations on a finite-dimensional $\k$-vector space $V$; for that group context will show whether the group scheme or the collection of linear transformations on $V$ (``$\GL(V)(\k)$") is meant.

For (Zariski) closed subgroups $G$, $N$ of $\GL(V)$, we write $G.N$ for the compositum of $G$ and $N$, i.e., for the smallest closed subgroup of $\GL(V)$ containing $G$ and $N$.  If $N$ normalizes $G$ and $\k$ is algebraically closed, $(G.N)(\k) = G(\k)\cdot N(\k)$.

Throughout we use notions from the theory of semisimple groups as in \cite{SGA3}, \cite{Borel}, \cite{Hum:LAG}, \cite{Hum:Lie}, or \cite{Bou:g7}, such as roots and weights.  

\section{Irreducible representations and the closed orbit}  \label{orbit.sec}

We now describe the basic setup that will be used throughout the paper, providing details and examples for the convenience of the reader who is a non-specialist in semisimple groups.

Let $\Gt$ be a split semisimple linear algebraic group over a field $K$ and fix a representation $\rho \!: \Gt \ra \GL(V)$;  Table \ref{1dim} below lists some examples of pairs $(\Gt, V)$ that we will consider.
For notational simplicity, we focus on the image $G$ of $\Gt$ in $\GL(V)$.  This group is also split semisimple.  We will assume that $V$ is an irreducible representation or is a Weyl module in the sense of \cite[p.~183]{Jantzen}.  (If $\car \k = 0$, the two notions coincide.)  In either case, $\End_G(V) = \k$, see loc.\ cit.\ when $V$ is a Weyl module.

Fix a pinning of $G$ in the sense of \cite[\S{XXIII.1}]{SGA3} (called a ``framing" in \cite{Bou:g7}); this includes choosing a split maximal $K$-torus $T$, a set of simple roots $\D$ of $G$ with respect to $T$, and a corresponding Borel subgroup $B$.  Recall that $T^*$ is naturally included in the weight lattice and there are bijections between dominant weights in $T^*$, equivalence classes of irreducible representations of $G$, and equivalence classes of Weyl modules of $G$ \cite[II.2.4]{Jantzen}.
Put $\la \in T^*$ for the highest weight of $V$ and $\vp$ for a highest weight vector in $V$. 

The stabilizer of $\k\vp$ in $G$ contains the Borel subgroup $B$, so it is a parabolic subgroup $P$.  The orbit $\O$ of $\k \vp$ is identified with the projective variety $G/P$, so $\O$ is \emph{closed} in $\P(V)$.

\begin{defn} 
We call an element $x \in V$ is \emph{minimal} if $\k x$ belongs to $\O$.  
\end{defn}

\begin{eg}[exterior powers] \label{basic.wedge}
Take $\Gt = \SL_n$ and $V = \wedge^d (K^n)$ for some $d$ between 1 and $n$.  The group acts via $\rho(g)(v_1 \wedge v_2 \wedge \cdots \wedge v_d) = gv_1 \wedge gv_2 \wedge \cdots \wedge gv_d$ for $v_1, \ldots, v_d \in K^n$.  The image $G$ of $\Gt$ is equal to $\SL_n/\mu_e$ where $\mu_e$ is group scheme of $e$-th roots of unity for $e := \gcd(d, n)$.  For $T$ and $B$, we take the image in $G$ of the diagonal and upper-triangular matrices, respectively.
The only line stabilized by $B$ is the span of $\vp = e_1 \wedge \cdots \wedge e_d$, where $e_i$ denotes the element of $K^n$ with a 1 in the $i$-th position and zeros elsewhere.

The group $\Gt$ is of type $A_{n-1}$ and its Dynkin diagram $\D$ is 
\[
\begin{picture}(6,0.7)
\multiput(0.2,0.2)(1,0){5}{\circle*{\darkrad}}
\put(0.2,0.2){\line(1,0){2}}
\put(3.2,0.2){\line(1,0){1}}
\multiput(2.45,0.2)(0.25,0){3}{\circle*{0.05}}
\put(0.2,0.3){\makebox(0,0.4)[b]{$1$}}
\put(1.2,0.3){\makebox(0,0.4)[b]{$2$}}
\put(2.2,0.3){\makebox(0,0.4)[b]{$3$}}
\put(3.2,0.3){\makebox(0,0.4)[b]{$n-2$}}
\put(4.4,0.3){\makebox(0,0.4)[b]{$n-1$}}
\end{picture}
\]
where we have labeled each vertex with the number $i$ of the corresponding fundamental weight $\omega_i$ according to the numbering from \cite{Bou:g4}.  With respect to this numbering, 
  $\vp$ has weight $\la = \omega_d$.  The representation $V$ is irreducible because $\omega_d$ is minuscule \cite[II.2.15]{Jantzen}.

An element of $V$ is \emph{decomposable} if it can be written as $v_1 \wedge \cdots \wedge v_d$ for some $v_i \in K^n$.  As $\SL_n$ acts transitively on the $d$-dimensional subspaces of $K^n$, we conclude that the minimal elements in $V$ are the nonzero decomposable vectors.

In the special case $d = 2$, we may identify $V$ with the vector space $\Skew_n$ of $n$-by-$n$ alternating matrices --- i.e., skew-symmetric matrices with zeros on the diagonal (the extra condition is necessary if $\car \k = 2$) --- where $\SL_n$ acts via $\rho(g)v = gvg^t$.  Then $\vp$ corresponds to $E_{12} - E_{21}$, where $E_{ij}$ denotes a matrix with a 1 in the $(i,j)$-entry and zeros elsewhere.  From this, we see that the minimal elements are the alternating matrices of rank 2.
\end{eg}

\begin{eg}[``symmetric powers''] \label{basic.S}
Take $\Gt = \SL_n$ (with fundamental weights numbered as in the previous example) and take $V$ to be the Weyl module with highest weight $d\omega_1$ for some $d \ge 1$.  If $\car K$ is zero or $> d$, this representation is irreducible by, for example, \cite[p.~50]{Green:GL}, and in that case $V$ can be identified with the $d$-th symmetric power $S^d(\k^n)$ of the tautological representation.  The group $G$ is $\SL_n / \mu_e$ as in the previous example, and the highest weight line is spanned by $\vp = e^d_1$.  This representation has highest weight $\la = d\omega_1$.  Here, minimal elements are the $d$-th powers of nonzero elements of $K^n$.

When $d = 2$ and $\car \k \ne 2$, we may identify $V$ with the vector space $\Sym_n$ of $n$-by-$n$ symmetric matrices, where $\SL_n$ acts by $\rho(g)v = gvg^t$, $\vp$ corresponds to $E_{11}$, and the minimal elements are symmetric matrices of rank 1.
\end{eg}

Returning to the case of general $G$ and $V$, we have:

\begin{eg} \label{basis.eg}
The collection of minimal elements is nonempty and $G$-invariant, so it spans a $G$-invariant subspace of $V$ that contains $\vp$, hence it must be all of $V$.  Therefore, \emph{there is a basis of $V$ consisting of minimal elements}.
\end{eg}

\begin{lem} \label{FH}
If $V$ is irreducible, then $\O$ is the unique closed $G$-orbit in $\P(V)$.  If $V$ is a Weyl module, then $\O$ is the unique closed $G$-orbit in $\P(V)$ such that $V$ is spanned by the lines in the closed orbit.
\end{lem}

\begin{proof}
Suppose $\O'$ is a closed $G$-orbit in $\P(V)$ such that $V$ is spanned by the lines in the closed orbit.  (This holds trivially  for any closed orbit when $V$ is irreducible.)  By the Borel Fixed Point Theorem, there is a line $\k x$ in $\O'$ that is stabilized by $B$, hence $x$ is a weight vector for some dominant $\mu \in T^*$.  All the weights of the $G$-submodule of $V$ generated by $x$ are $\le \mu$ by \cite[2.13]{Jantzen}, but this submodule is all of $V$, so $\mu$ must equal $\la$ and $\k \vp = \k x$.
\end{proof}

The extra condition in Lemma \ref{FH} for Weyl modules $V$ is necessary.  Indeed, if there is an exact sequence of representations $1 \ra A \ra V \ra B \ra 1$ where $A$ and $B$ are irreducible, then the closed orbit in $\P(A)$ gives a closed orbit in $\P(V)$ distinct from $\O$.  This occurs, for example, when $G = \SL_2$ over a field $\k$ of characteristic 3 and $V$ is the Weyl module with highest weight 3, in which case we additionally have that the two closed orbits are isomorphic as varieties (to $\P^1$).

\begin{cor} \label{rk1.fix}
Every element of $\GL(V)$ that acts trivially on $\O$ is a scalar matrix.
\end{cor}

\begin{proof}
Fix a (commutative) $\k$-algebra $R$ and an element $g \in \GL(V)(R)$ in the fixer of the closed $G$-orbit.  By hypothesis, for every $[x]$ in the closed $G$-orbit, there is a scalar $c_x \in R^\times$ such that $gx = c_x x$, i.e., a morphism of varieties $G/P \ra \Gm$ given by $[x] \mapsto gx / x$.  As $G/P$ is projective and $\Gm$ is affine, the image must be a point.  That is, for every minimal $x$, we have $gx = cx$ for some $c \in R^\times$ and Example \ref{basis.eg} gives the claim.
\end{proof}

This corollary has occasionally been proved in a few special cases in the literature on linear preserver problems, see for example the end of \cite{Dieu:LPP}, Prop.~8 in \cite{Jac:J3}, or Cor.~6.3 in \cite{Ferr:strict}.

\begin{defn} \label{pres.min}
A linear transformation $T \!: V \ra V$ \emph{preserves minimals} if for every minimal element $x \in V \ot \kalg$ --- where $\kalg$ denotes an algebraic closure of $\k$ --- $T(x)$ is also a minimal element. 

The ``$\ot \kalg$" is necessary here, because this is the condition we use in our proofs.  It is possible that an element of $\GL(V)$ can preserve the minimal elements of $V$ but fail to preserve the minimal elements of $V \ot \kalg$ even in the very nice case where $\k = \R$, see e.g.\ \cite[Example 1]{Zhang:2}.
\end{defn} 

\begin{prop} \label{ChanLim}
If $T$ is a linear transformation that preserves minimals, then $T$ is invertible.
\end{prop}

\begin{proof} 
We recall the argument from \cite[p.~322]{ChanLim:symm2}.  The set $X$ of minimal elements in $V$ has closure $\Xb = X \cup \{ 0 \}$, an irreducible subvariety of $V$.  Note that  $T(\Xb)$ is a closed subvariety of $\Xb$, as can be seen by considering the morphism induced by $T$ on the image $G/P$ of $X$ in $\P(V)$.  The fiber of $T \!: \Xb \ra T(\Xb)$ over 0 is just $\{ 0 \}$, and we deduce that $\dim T(\Xb) = \dim \Xb$ \cite[Th.~4.1]{Hum:LAG}, hence $T(\Xb) = \Xb$.  As $\Xb(\k)$ contains a spanning set for $V$ (Example \ref{basis.eg}), the claim is proved.
\end{proof}

In older times, this was proved by hand for each choice of $G$ and $V$, see for example \cite{Westwick:Grassmann}.

\section{The normalizer of $G$ in $\GL(V)$} \label{norm.sec}

The goal of this section is to prove a technical statement about the structure of the normalizer of $G$ in $\GL(V)$, Proposition \ref{AutD} below.  We maintain the notation and hypotheses of section \ref{orbit.sec}.  

Write $\Aut(\D)$ for the automorphism group of the Dynkin diagram of $G$.  (This is an abuse of notation in that we have already defined $\D$ to be the set of simple roots, i.e., the vertex set of the Dynkin diagram.)
We write $\Aut(\D, \la)$ for the subgroup of $\Aut(\D)$ fixing $\la$.

\begin{eg} \label{basic.eg2}
Returning to Examples \ref{basic.wedge} and \ref{basic.S}, the group $\Aut(\D)$ acts on the weights by permuting the fundamental weights according to its action on the diagram; we find that $\Aut(\D, \la) = 1$ for $V$ with highest weight $d \omega_1$, but for $V = \wedge^d (K^n)$ we have $\Aut(\D,\la) = 1$ for $n \ne 2d$, and $\Zm2$ for $n = 2d$ (in particular for $\SL_4$ acting on the 4-by-4 alternating matrices).
\end{eg}

Write $\Aut(G, \la)$ for the inverse image of $\Aut(\D, \la)$ under the map $\Aut(G) \ra \Aut(\D)$.
To spell this out, given an automorphism $\phi$ of $G$, we may compose it with conjugation by an element of $G$  to produce an element $\phi'$ such that $\phi'(T) = T$ and $\phi'(B) = B$.  The automorphism $\phi'$ is   defined up to conjugation by an element of $T$, so the action of $\phi'$ on $T^*$ is uniquely determined by $\phi$.  Then $\Aut(G, \la)(\k)$ is the collection of $\phi \in \Aut(G)(\k)$ such that $\phi'(\la) = \la$.

The pinning induces a homomorphism $i$ embedding $\Aut(\D)$ in the automorphism group of the simply connected cover of $G$ \cite[XXIII.4, Th.~4.1]{SGA3}.  Further, writing $Z$ for the center of $G$, we have:
\begin{prop} \label{AutD}
The map $i$ induces a homomorphism $\Aut(\D, \la) \ra \Aut(G, \la)$ and an injection $\gamma$ 
such that the diagram
\[
\xymatrix{&&\Aut(\D, \la) \ar[d]_\gamma \ar[dr]^i \\
1 \ar[r] & \Gm \ar[r] & N_{\GL(V)}(G) \ar[r]^\Int & \Aut(G, \la) \ar[r] & 1 
}
\]
commutes, the horizontal sequence is exact, and $\Int$ is surjective on $\k$-points.  Furthermore, $N_{\GL(V)}(G)$ is smooth and $\gamma$ identifies $N_{\GL(V)}(G)$ with $\left( (\Gm \times G)/Z\right) \rtimes \Aut(\D, \la)$.
\end{prop}

In the statement, we wrote $\Int$ for the map such that $\Int(n)(g) = ngn^{-1}$ for $n \in N_{\GL(V)}(G)(\k)$ and $g \in G(\k)$.  For  a definition of short exact sequences of affine group schemes and their basic properties, see for example \cite[p.~341]{KMRT}.

\begin{proof}
For the purpose of this proof, write $\Gt$ for the simply connected cover of $G$ and $\Zt$ for its center.  For $\pi \in \Aut(\D, \la)$, $i(\pi) \in \Aut(\Gt)$ normalizes $\ker \la|_{\Zt}$, which is the kernel of $\Gt \ra G$. 
Hence $i(\pi)$ induces an automorphism of $G$.  As $i$ is a section of the natural homomorphism $\Aut(\Gt) \ra \Aut(\D)$, it is also a section of the natural homomorphism $\Aut(G, \la) \ra \Aut(\D, \la)$.  It follows from this discussion that $i$ identifies $\aut(G, \la)$ with $(G/Z) \rtimes \Aut(\D, \la)$.  

We claim that $(G/Z)(\k)$ is in the image of $N_{GL(V)}(G)(\k)$.  Indeed, the normalizer contains $G$ and the scalar matrices $\Gm$, and these two groups have intersection $Z$.  This gives an exact sequence
\begin{equation} \label{AutD.1}
\begin{CD}
1 @>>> \Gm @>>> (\Gm \times G)/Z @>\Int>> G/Z @>>> 1.
\end{CD}
\end{equation}
Applying fppf cohomology gives an exact sequence
\[
\begin{CD}
((\Gm \times G)/Z)(\k) @>\Int>> (G/Z)(\k) @>>> H^1(\k, \Gm),
\end{CD}
\]
where the last term is zero by Hilbert's Theorem 90.  As the first term is contained in $N_{\GL(V)}(G)(\k)$, we have verified the claim.

For $n \in N_{\GL(V)}(G)(\k)$, $\Int(n)$ is an automorphism of $G$ and modifying it by conjugation by an element of $G(\k)$, we may assume that $\Int(n)$ normalizes $B$ and $T$.  As $\Int(n)$ defines an equivalence of the irreducible representations or Weyl modules with highest weights $\la$ and $^n \la$, we deduce that $\Int(n)$ belongs to $\Aut(G, \la)(\k)$.  Running this argument backwards shows that $\Int$ is surjective on $\k$-points.  This completes also the proof that the sequence is exact and that $N_{\GL(V)}(G)$ is smooth (because $\Gm$ and $\Aut(G, \la)$ are).

To construct $\g$, we take $\pi \in \Aut(\D, \la)(\k)$.  The element $n$ such that $\Int(n) = i(\pi)$ is determined up to a factor in $\kx$; we pick $n$ so that $n\vp = \vp$ and put $\g(\pi) := n$.  To verify that it is a homomorphism, note that $\Int(\g(\pi_1 \pi_2)) = \Int(\g(\pi_1) \g(\pi_2))$, so $\g(\pi_1 \pi_2)$ and $\g(\pi_1) \g(\pi_2)$ differ by at most a factor in $\kx$.  But both elements of $\GL(V)$ fix $\vp$, so they are equal.

For the final claim, note that if $\pi \in \Aut(\D, \la)$ is such that $\gamma(\pi)$ is in the identity component of $N_{\GL(V)}(G)$, then $\Int \gamma(\pi) = i(\pi)$ belongs to the identity component of $\Aut(G, \la)$, i.e., to $G/Z$, and we conclude that the semidirect product $N'$ of $N_{\GL(V)}(G)^\circ$ and $\gamma(\Aut(\D, \la))$ is identified with a subgroup of the normalizer.  Furthermore, writing $\pi_0$ to mean the component group, we have 
\[
\gamma(\Aut(\D, \la)) = \pi_0(N') \subseteq \pi_0(N_{\GL(V)}(G)) = \pi_0(\Aut(G, \la)) = i(\Aut(\D, \la)),
\]
so $N'$ equals $N_{\GL(V)}(G)$.  Exactness of \eqref{AutD.1} completes the proof.
\end{proof}

\begin{cor} \label{norm.gen}
If $\k$ is algebraically closed, then $N_{\GL(V)}(G)(K)$ is generated by $G(\k)$, $\k^\times$, and $\gamma(\Aut(\D, \la))$.$\hfill\qed$
\end{cor}

In many of the examples considered below, the following holds:
\begin{equation} \label{Lt.hyp}
\parbox{4.5in}{There is a connected reductive group $\Lt$ such that $\Lt$ has semisimple part $\Gt$ and $\rho$ extends to a homomorphism $\rt \!: \Lt \ra \GL(V)$ such that $\im \rt \supseteq \Gm$ and $\ker \rt$ is a split torus.}  
\end{equation}
This allows us to make the more attractive statement, which holds with no hypotheses on $\k$:
\begin{cor} \label{Lt.norm}
Assuming \eqref{Lt.hyp}, $N_{\GL(V)}(G)(\k)$ is the subgroup of $\GL(V)$ generated by $\rt(\Lt(\k))$ and $\gamma(\Aut(\D, \la))$.
\end{cor}

\begin{proof}
The sequence $1 \ra \ker \rt \ra \Lt \xrightarrow{\rt} N_{\GL(V)}(G)^\circ \ra 1$ is exact by the preceding corollary, so $\Lt(K) \ra N_{\GL(V)}(G)^\circ \ra H^1(\k, \ker \rt)$ is exact.  But the last term is 1 by Hilbert 90 because $\ker \rt$ is a split torus.
\end{proof}

\section{Linear transformations preserving minimal elements} \label{min.sec}

We maintain the notation and hypotheses of section \ref{orbit.sec}.  By Proposition \ref{ChanLim}, every linear transformation of $V$ that preserves minimals belongs to $\GL(V)$, so the collection of such is the sub\emph{group} $\Stab_{\GL(V)}(\O)$ of $\GL(V)$ stabilizing the closed $G$-orbit $\O \subseteq \P(V)$.

\begin{eg} \label{norm.eg}
If $n \in \GL(V)$ normalizes $G$, then for every minimal $x$, the $G$-orbit of $nx$ in $\P(V)$ is closed and spans $V$, hence $nx$ is also minimal by Lemma \ref{FH}.  That is, $N_{\GL(V)}(G)$ is contained in $\Stab_{\GL(V)}(\O)$.
\end{eg}

Under a technical hypothesis spelled out in Definition \ref{dem.def}, we can say that this containment is an equality. 
Recall that $P$ is the parabolic subgroup of $G$ stabilizing the highest weight line $\k\vp$.

\begin{thm} \label{rk1}
$\Stab_{\GL(V)}(\O) = N_{\GL(V)}(G)$ 
if $P$ is not exceptional.
\end{thm}

The proof will come at the end of the section.

\begin{defn} \label{dem.def}
Following \cite{Dem:aut}, we define:
 \begin{enumerate}
\item If $G$ is simple, then $P$ is \emph{exceptional} in the following cases:
   \begin{enumerate}
   \item \label{Dem.Sp} $G$ has type $C_\ell$ with $\ell \ge 2$ and $P$ has Levi subgroup of type $C_{\ell -1}$.
   \item \label{Dem.HSpin} $G$ has type $B_\ell$ with $\ell \ge 2$ and $P$ has Levi subgroup of type $A_{\ell - 1}$.
   \item \label{Dem.G2} $G$ has type $G_2$ and $P$ is the stabilizer of the highest weight vector in the 7-dimensional fundamental Weyl module.
      \end{enumerate}
\item If $G$ is not simple, we write its adjoint group as $\Gb_1 \times \cdots \times \Gb_r$ where each $\Gb_i$ is simple.  We say that $P$ is exceptional if at least one of its images in $\Gb_1, \ldots, \Gb_r$ is exceptional.
\end{enumerate}
\end{defn}

All three cases in \ref{dem.def}(1) give genuine exceptions to Theorem \ref{rk1}.  Item \eqref{Dem.Sp} includes the case where $G = \Sp_{2\ell}$ for $\ell \ge 2$ and $V$ is the natural representation.  In that case, every nonzero vector is a minimal element, so $\Stab_{\GL_{2\ell}}(\O)$ is all of $\GL_{2\ell}$, but $N_{\GL_{2\ell}}(\Sp_{2\ell})$ is $\Gm.\Sp_{2\ell}$, which has dimension only $2\ell^2 + \ell+1$.  Item \eqref{Dem.HSpin} is addressed in \S\ref{ignore}.  Item \eqref{Dem.G2} includes the case where $G$ is the automorphism group of the split octonions and $V$ is the space of trace zero octonions.  In that case, the closed $G$-orbit is the quadric in $\P(V)$ defined by the quadratic norm form $q$ \cite[9.2]{CG}, hence $\Stab_{\GL(V)}(\O)$ is the group of similarities of $q$; this has dimension 22 as opposed to $\dim \Gm.G = 15$.

\begin{rmk*}
Demazure includes a fourth item in his version of \ref{dem.def}, namely that $G \ne 1$ and $P = G$, which would appear as (1d) in \ref{dem.def} above.  But this case cannot occur here due to our assumption that the representation $V$ is faithful.\end{rmk*}

\begin{proof}[Proof of Theorem \ref{rk1}] We abbreviate $N := N_{\GL(V)}(G)$ and $S := \Stab_{\GL(V)}(\O)$.
Consider the diagram
\[
\xymatrix{
1 \ar[r] &\Gm \ar[r] \ar@{=}[d]& N^\circ \ar[r]^{\Int} \ar@{^(->}[d]& \Aut(G)^\circ \ar[r] \ar[d]^c& 1 \\
1 \ar[r] & \Gm \ar[r]& S^\circ \ar[r] &\Aut(G/P)^\circ \ar[r] & 1
}
\]
where the middle vertical arrow is the natural inclusion (as in Example \ref{norm.eg}) and $c$ is given by
\[
c(\Int(x))gP = xgP \quad \text{for $g \in G(\k)$ and $\Int(x) \in \Aut(G)^\circ(\k)$.}
\]
The top sequence is exact.  As $P$ is not exceptional, \cite[Th.~1]{Dem:aut} gives that $c$ is an isomorphism.  From this and Corollary \ref{rk1.fix} we see that the bottom sequence is exact.
In particular, $S^\circ$ is smooth because $\Aut(G/P)^\circ$ and $\Gm$ are.  Further,
\[
1 + \dim \Aut(G)^\circ = \dim N^\circ \le \dim S^\circ = 1 + \dim \Aut(G)^\circ,
\]
so we have $S^\circ = N^\circ$.

Recall from Proposition \ref{AutD} that $N^\circ$ is reductive with semisimple part $G$, hence $G$ is a characteristic subgroup of $S^\circ$.  As $S$ normalizes $S^\circ$, we deduce that $S$ normalizes $G$.
\end{proof}

We now apply Theorem \ref{rk1} and Proposition \ref{AutD} to compute $\Stab_{\GL(V)}(\O)$ in various specific cases. 
 For the following result compare \cite[Cor.~2]{Lim:symm}, \cite[Th.~6.5]{Wa:det}, \cite[Th.~11]{Wa:symm1}, or \cite[Cor.~6.2] {Guralnick:LPP2} (with $\k$ algebraically closed).

\begin{cor}[symmetric matrices] \label{symm.rk1}
Suppose $\k$ has characteristic $\ne 2$.  Every invertible linear transformation of $\Symm_n(\k)$ that sends rank $1$ matrices to rank $1$ matrices is of the form
 \begin{equation} \label{skew.1}
X \mapsto rPXP^t \quad \text{for some $r \in \kx$ and $P \in \GL_n(\k)$.}
\end{equation}
\end{cor}

\begin{proof}
Take $\Gt = \GL_n$ as in Example \ref{basic.S}.  In the notation of \eqref{Lt.hyp}, take $\Lt = \Gm \times \GL_n$ and for $(r, P) \in \kx \times \GL_n(\k)$, define $\rt(r, P)$ as in \eqref{skew.1}.  For every commutative $\k$-algebra $R$, the set of $R$-points of $\ker \rt$ is $\{ (t^2, t^{-1}) \mid t \in R^\times \}$.  As $\Aut(\D, \la) = 1$, combining Corollary \ref{Lt.norm} and Theorem \ref{rk1} gives the claim.
\end{proof}

For the next result compare  \cite[Th.~3]{MarcusWestwick} (with $\k = \R$), \cite[Th.~5.5]{Wa:det}, or \cite[Cor.~7.3]{Guralnick:LPP2}.

\begin{cor}[alternating matrices] \label{skew.rk1} 
For $n \ge 2$ and $n \ne 4$, every invertible linear transformation of $\Skew_n(\k)$ that sends rank $2$ matrices to rank $2$ matrices is of the form \eqref{skew.1}.
If $n = 4$, then every invertible linear transformation of $\Skew_n(\k)$ that sends rank $2$ matrices to rank $2$ matrices is as in \eqref{skew.1} or is 
\begin{equation} \label{skew.2}
X \mapsto rPX^* P^t \quad \text{for some $r \in \kx$, $P \in \GL_n(\k)$,} 
\end{equation}
and 
\begin{equation} \label{skew.star}
\left( \begin{smallmatrix}
0&x_1&x_2&x_3\\
-x_1&0&x_4&x_5\\
-x_2&-x_4&0&x_6\\
-x_3&-x_5&-x_6&0
\end{smallmatrix} \right)^* = 
\left( \begin{smallmatrix}
0&x_1&-x_2&-x_4\\
-x_1&0&-x_3&-x_5\\
x_2&x_3&0&x_6\\
x_4&x_5&-x_6&0
\end{smallmatrix} \right).
\end{equation}
\end{cor}

The map $*$ is a Hodge star operator, which is not uniquely determined.  Said differently, one can replace $*$ with its composition by any map as in \eqref{skew.1}.  Therefore, one finds slightly different formulas in other sources, such as \cite[p.~921]{MarcusWestwick} and \cite[p.~15]{Pierce:survey}. 

\begin{proof}[Proof of Corollary \ref{skew.rk1}]
We use the same $\Lt$ and $\rt$ from the proof of the preceding corollary, substituting $\Skew_n(\k)$ for $V$.
 If $n \ne 4$, $\Aut(\D, \la) = 1$ and the proof is complete.  Otherwise $n = 4$, $\Aut(\D, \la) = \Zm2$ and we are tasked with finding the image of the nonidentity element $\pi$ of $\Aut(\D, \la)$ under the map $\gamma$ from Proposition \ref{AutD}.  The element $*$ of $\GL(V)$ fixes $\vp = E_{12} - E_{21}$.  Furthermore, one checks that $\Int(*)$ normalizes the maximal torus $T$ and permutes the root subgroups (described in \cite[\S{VIII.13.3}]{Bou:g7}) as indicated by the action of $\pi$ on $\D$; it follows that $*$ normalizes $G$, hence $\gamma(\pi) = *$.
\end{proof}

For the next result compare \cite{Hua:Gr}; \cite[Th.~1]{MarcusMoyls:tensor} or \cite[Th.~1]{Minc:rk1} (for $\k$ algebraically closed of characteristic zero); or \cite[Th.~3.5]{Wa:det}.
\begin{cor}[rectangular matrices] \label{rect.rk1} 
For $m,n \ge 2$ and $m \ne n$, every invertible  linear transformation of the $m$-by-$n$ matrices with entries in $\k$ that sends rank $1$ matrices to rank $1$ matrices is of the form
\begin{equation} \label{rect.1}
X \mapsto AXB \quad \text{for some $A \in \GL_m(\k)$ and $B \in \GL_n(\k)$.}
\end{equation}
For $n = m \ge 2$, every invertible linear transformation of $M_n(K)$ is of the form \eqref{rect.1} or is
\begin{equation} \label{rect.2}
X \mapsto AX^t B \quad \text{for some $A, B \in \GL_n(\k)$}.
\end{equation}
\end{cor}

\begin{proof}[Sketch of proof]
Here one takes $\Gt := \SL_m \times \SL_n$ and $\Lt := \GL_m \times \GL_n$ acting on the space $V$ of $m$-by-$n$ matrices via $\rt(A,B)X = A X B^t$ and imitates otherwise the proofs of Corollaries \ref{symm.rk1} and \ref{skew.rk1}.
\end{proof}

\begin{eg}[homogeneous polynomials of degree $d$] \label{S.rk1}
Assume $\chr \k = 0$ or $> d$.  Viewing $K^n$ as the dual of a vector space, the representation $V$ from Example \ref{basic.S} becomes the vector space 
 of homogeneous polynomials of degree $d$ in $n$ variables.   As $\Aut(\D, \la) = \{ \Id_\D \}$, Theorem \ref{rk1} gives: \emph{the collection of linear transformations of $V$ that preserve the set of $d$-th powers of nonzero linear forms is the compositum of $G$ and the scalar matrices in $\GL(V)$.}  Compare \cite[Th.~10.5.5]{Shaw2}.
 \end{eg}

\begin{eg}[exterior powers] \label{wedge.rk1}
Take $\Gt = \SL_n$ and $V = \wedge^d \k^n$ for some $1 \le d < n$ as in Example \ref{basic.wedge}.  Theorem \ref{rk1} gives for $d \ne n/2$: \emph{the collection of linear transformations of $V$ that preserve the set of nonzero decomposable vectors is the compositum of $\SL_n$ and the scalar transformations.}  
(In case $d = n/2$, every linear transformation that preserves the decomposable vectors is as in the previous sentence, or is the composition of such a transformation with a Hodge star operator.)  Compare \cite[3.1, 3.2]{Nemitz}, \cite{Westwick:Gr}, or \cite[Cor.~7.3]{Guralnick:LPP2}.
\end{eg}

The hypothesis on $V$ in Theorem \ref{rk1}---that $P$ is not exceptional---is weak enough that many other examples can also be treated readily.  For example, one can recover the stabilizer of the decomposable tensors  in $\ot_{i=1}^r \k^{n_i}$  as in \cite[Th.~3.8]{Westwick:tensor} (which assumes $\k$ algebraically closed).  We also have the following:

\begin{eg}[pure spinors] \label{spinors}
Take $\Gt = \Spin_{2n}$ for some $n \ge 3$ (so $\Gt$ is of type $D_n$) and take $V$ to be a half-spin representation as defined in \cite{Chev} or \cite[\S{VIII.13.4(IV)}]{Bou:g7}.  This representation is injective (i.e., $G = \Gt$) if $n$ is odd, and has kernel $\mu_2$ if $n$ is even; in this latter case, the image $G$ is called a half-spin or semi-spin group.  The minimal elements in $V$ are the \emph{pure spinors} as defined in \cite[\S3.1]{Chev}.

The representation $V$ is irreducible (regardless of the characteristic of $\k$) because it is minuscule. As $\Aut(\D, \la) = \{ \Id_\D \}$, Theorem \ref{rk1} gives: \emph{the collection of linear transformations of $V$ that preserve the set of pure spinors is the compositum of $G$ and the scalar matrices.} 
\end{eg}

\begin{eg}[minimal nilpotents] \label{nilpotents}
Let $\Gt$ be a split simple and simply connected group and take $V = \Lie(\Gt)$.  This is a Weyl module for $\Gt$, and it is irreducible if $\car \k$ is very good for $\Gt$.  The minimal elements in $\Lie(\Gt)$ are called \emph{minimal nilpotents}, cf.~\cite[4.3.3]{CMcG}.

As $\la$ is the highest root, $\Aut(\D, \la) = \Aut(\D)$.  Theorem \ref{rk1} gives that the collection of linear transformations of $\Lie(\Gt)$ that preserve the minimal nilpotents is the compositum of the adjoint group $G$, the scalar transformations, and a copy of $\Aut(\D)$.
\end{eg}

To summarize what we observed in this section:\! it is not difficult to see that the minimal elements are preserved by the normalizer of $G$ in $\GL(V)$ (Example \ref{norm.eg}).  Using Demazure's description of  the automorphism group of projective homogeneous varieties, we showed that the normalizer is exactly the group of linear transformations preserving the minimal elements (Theorem \ref{rk1}).  From this and \S\ref{norm.sec}, one can read off the group of linear transformations that preserve the minimal elements in many cases.
\section{Interlude: non-split groups} \label{nonsplit.sec}

So far, we have assumed that the group $G$ is split.  We now explain how to remove this hypothesis.  Suppose for the duration of this section that $G$ is a semisimple group over $\k$ with a faithful representation $\rho \!: G \ra \GL(V)$, and that, after base change to a separable closure $\ksep$ of $\k$, $\rho$ is irreducible or a Weyl module.

We can fix a pinning of $G \times \ksep$, a highest weight vector $\vp \in V \ot \ksep$, a parabolic $P := \Stab_{G \times \ksep}(\ksep \vp)$, and a closed orbit $\O \in \P(V) \times \ksep$ as in \S\ref{orbit.sec}.  
\begin{prop} \label{nonsplit}
The closed $G$-orbit $\O$ is defined over $\k$ and $\Stab_{\GL(V)}(\O) = N_{\GL(V)}(G)$.
\end{prop}

\begin{proof}
For $\s \in \Gal(\ksep/k)$, the action of $G$ on $V$ commutes with $\s$, so $\s(\O)$ is a closed $G$-orbit in $\P(V)$ whose elements span $V \ot \ksep$, ergo $\s(\O(\ksep)) = \O(\ksep)$.  By the Galois criterion for rationality \cite[AG.14.4]{Borel}, $\O$ is defined over $\k$.  The group schemes $N_{\GL(V)}(G)$ and $\Stab_{\GL(V)}(\O)$ are both defined over $\k$, and the claimed equality is by Theorem \ref{rk1}.
\end{proof}

Suppose now that the representation $V \ot \ksep$ is as in Table \ref{1dim}.  We will prove in Propositions \ref{RRS.prop}, \ref{FTS.prop}, and \ref{dieu.prop} that
\[
\Stab_{\GL(V)}(f) \subset \Stab_{\GL(V)}(\{ f = 0 \}) = N_{\GL(V)}(G)
\]
as group schemes over $\ksep$, and it follows from Proposition \ref{nonsplit} that these relationships also hold over $\k$.

\section{Representations with a one-dimensional ring of invariants} \label{LPP.1}

We have completed our study of linear transformations that preserve minimal elements, and we now move on to considering linear preserver problems (LPPs) as described in the introduction.  
We maintain the notation of section \ref{orbit.sec}, so $\Gt$ is a split semisimple algebraic group over the field $\k$ and $\rho \!: \Gt \to \GL(V)$ is an irreducible representation or a Weyl module.  The $d = 2$ cases from Examples \ref{basic.wedge} and \ref{basic.S} are  special in that the ring $\k[V]^G$ of $G$-invariant polynomial functions on $V$ equals $\k[f]$ for a nonconstant homogeneous polynomial $f$.  (We say that $f$ is $G$-invariant if every $g \in G(\kalg)$ preserves $f$, where we use the typical algebraist's definition that an element $g \in G(\k)$ \emph{preserves} $f$ if $f \circ g = f$ as polynomials.)

The basic facts about this situation are given by the following proposition, which is well known for $\k = \C$, see e.g.\ \cite[Prop.~12]{Popov:14}. The quotient $V/G$ is defined to be the variety $\Spec \k[V]^G$.

\begin{prop} \label{dim1}
The following are equivalent:
\begin{enumerate}
\item $\dim V/G = 1$.
\item There is a dense open $G(\kalg)$-orbit in $\P(V)(\kalg)$ but not in $V \ot \kalg$.
\item $\k[V]^G = \k[f]$ for some homogeneous $f \in \k[V] \setminus \k$.
\item $V/G$ is isomorphic to the affine line $\aff^1$.
\end{enumerate}
\end{prop}

\begin{proof}
Assuming (1), Theorem 4 in \cite{Schinzel} gives that the fraction field of $\k[V]^G$ is $\k(f)$ for some homogeneous $f$, and as in the proof of that theorem we deduce (3).
(3) implies (4) because 
the polynomial $f$ (as an element of $\k[V]$) is transcendental over $\k$, and (4) trivially implies (1).

Now suppose (2).   Recall that 
 there is a $G$-invariant  dense subset $U$ of $V \ot \kalg$ such that two elements of $U$ have the same image in $V/G$ iff they are in the same $G(\kalg)$-orbit.  So, if  $L$ is a line in $V$ that is in the open orbit in $\P(V)(\kalg)$, then $G(\kalg)\cdot L$ contains a nonempty open subset of $V$, hence contains a nonempty open subset of $U$; it follows that the map $L \ra V/G$ is dominant, hence that $\dim V/G$ is 0 or 1.  But for an orbit $X$ of maximal dimension in $V$, we have $\dim V/G = \dim V - \dim X$, so if $\dim V/G$ is 0, there is a dense orbit in $V \ot \kalg$.  (1) is proved.

Finally suppose (3) holds; we prove (2).   The map $f \!: V \ot \kalg \ra \kalg$ is $G$-invariant and nonconstant, so  there is no dense orbit in $V \ot \kalg$.  As $f$ is homogeneous and separates the $G(\kalg)$-orbits in $U$, it follows that the dense image of $f^{-1}(\kalg^\times) \cap U$ in $\P(V)$ is a single (open) $G(\kalg)$-orbit.\end{proof}

Representations where the conditions in the proposition hold are closely related to the prehomogeneous vector spaces studied in \cite{SK}, the $\theta$-groups studied by Vinberg as in \cite{PoV}, and the internal Chevalley modules from \cite{ABS}.

All pairs $(G, V)$ with one-dimensional ring of invariants, $G$ simple, and $\k = \C$ are listed on pages 260--262 of \cite{PoV}.  In sections \ref{RRS.sec} and \ref{FTS.sec}, we will solve the LPP for the representations listed in Table \ref{1dim}.  That table does not include all possibilities from \cite{PoV}, and we make some remarks about the remaining entries in section \ref{ignore}.

\subsection*{The representations in Table \ref{1dim} are irreducible} We now verify that the representations in Table \ref{1dim} are irreducible.  If $\k = \C$ then this is well known, so we may assume that $\car \k$ is a prime.  We fix a particular $(G, V)$ from Table \ref{1dim} for consideration and write $C \subseteq \{ 2, 3\}$ for the set of excluded characteristics for that line of the table.  There is a simply connected split group $H$ (of type $C_n$, $D_n$, $E_7$, $A_{2n-1}$, $G_2$, $F_4$, $E_6$, $E_7$, or $E_8$ respectively) and a homomorphism $\ell \!: \Gm \ra H$ such that the centralizer of $\im \ell$ in $H$ (which is reductive) has derived subgroup $\Gt$.  Also, the 1-eigenspace of $\ell$ in $\Lie(H)$---i.e., the subspace 
\[
\Lie(H)_1 := \{ x \in \Lie(H) \mid \text{$\Ad \ell(t) \cdot x = tx$ for all $t \in \Gm$} \}
\]
--- is isomorphic to $V$ in such a way that the adjoint action $\Gt \ra \GL(\Lie(H)_1)$ is identified with the representation $\rho \!: \Gt \ra \GL(V)$.  It follows in particular that $V$ is an irreducible representation of $G$ and that there is an open $G$-orbit in $\P(V)$, by \cite{ABS}, using that $\car \k \not\in C$.
As $V$ is not exceptional in the sense of Definition \ref{dem.def}, Theorem \ref{rk1} applies and $N_{\GL(V)}(G) = \Stab_{\GL(V)}(\O)$.

\subsection*{The polynomial $f$} The machinery in the previous paragraph also gives more.  Let $R$ be a subring of $\Q$ such that the set $C$ of excluded primes is contained in $R^\times$.  There is a group $H$ and homomorphism $\ell$ defined over $R$ so that their base change to $K$ is as in the preceding paragraph.  The paper \cite{Sesh:GR} shows that the quotient variety $V/\Gt$ is defined over $R$.  We have
\[
\dim (V \times K) / (\Gt \times K) = \dim (V/\Gt) \times K = \dim (V/\Gt) \times \C = \dim (V \times \C) / (\Gt \times \C),
\]
where the 
last number is 1 by our choice of $(G, V)$, see \cite{PoV} or \cite{Kac:nil}.  That is, the equivalent conditions of Proposition \ref{dim1} hold.  

One can choose an $f$ that generates $\C[V]^G$ and comes from $R[V]^G$; if the image of this $f$ in $\k[V]$ is nonzero, then it generates $\k[V]^G$.  Therefore, to determine a formula for $f$ or its degree, it suffices to do so over $\C$.  The degree can be looked up in \cite{PoV} or \cite[Table II]{Kac:nil} or can be calculated using \cite[p.~65, Prop.~15]{SK}.

\begin{eg}[binary cubics] \label{binarycubics.f}
Suppose $\car \k \ne 2, 3$ and consider the vector space $V$ of binary cubic forms; it is the irreducible representation $V = S^3((K^2)^*)$ of $G = \SL_2$ from line 5 of Table \ref{1dim}.
In the notation of the three preceding paragraphs, one takes $H$ split of type $G_2$ and $R = \Z[\frac12, \frac13]$.  The maps that send $(x,y)$ to $x^3, x^2y, xy^2, y^3$ are a basis for the $R$-module $V$, and a formula for $f$ is given in \cite[\S46, (10)]{Weber} or \cite[(14.33)]{Gurevich}:
\[
f(a_0x^3 + a_1x^2y+ a_2xy^2 + a_3y^3) = a_1^2 a_2^2 + 18a_0a_1a_2a_3 - 4a_0a_2^3-4a_1^3a_3 - 27 a_0^2 a_3^2.
\]
This is the \emph{discriminant} of the cubic form.  

We remark that binary cubic forms (and $f$) can be identified with cubic algebras (and their discriminant algebras) as described in \cite[\S4]{GanGrossSavin} or \cite{HoffMorales}.  Furthermore, this representation is irreducible also in case $\car \k = 2$ by Steinberg's tensor product theorem \cite[II.3.17]{Jantzen}.  We ignore these variations below.
\end{eg}

\begin{eg} \label{blackholes.f}
The group $\Gt = \SL_2 \times \SO_n$ for $n \ge 4$ acts naturally on $V = \k^2 \ot \k^n$.  If $n$ is odd, assume $\car \k \ne 2$.  Then $V$ is irreducible and in the notation of this section, one takes $H = \Spin_{n+4}$ and $R = \Z$ or $\Z[\frac12]$. We may identify $V$ with the 2-by-$n$ matrices so that $\Gt$ acts via $\rho(g_1, g_2) X = g_1 X g_2^t$.  There is a symmetric $S \in \GL_n(\k)$ so that the $\k$-points of $\SO_n$ are the $g_2 \in \SL_n(\k)$ such that $g_2^t S g_2 = S$.  It follows that the polynomial map $f \!: V \ra \k$ defined by
$f(X) := \det(XSX^t)$ is invariant under $\Gt$.  It generates $\C[V]^G$ as argued in \cite[pp.~109, 110]{SK}, so $\k[V]^G = \k[f]$.  

In the smallest case $n = 4$, one can equivalently take $\Gt = \SL_2 \times \SL_2 \times \SL_2$ and $V = \k^2 \ot \k^2 \ot \k^2$.  In that case, $f$ is Cayley's hyperdeterminant defined in \cite{Cayley:hyper}.  It appears, for example, in quantum information theory to measure the entanglement of a 3-qubit system \cite{MiyakeWadati}.
\end{eg}
\section{Transformations that preserve minimal elements and $f$} \label{LPP.soln}

\begin{table}[bp]
\[
\begin{array}{|r|ccc|cc|c|}\hline
\#&\Gt\upstrut{3}&V&\dim V&f&\deg f&\chr \k \\ \hline
1&\SL_n&S^2(\k^n)&\binom{n}{2}+n&\det&n&\ne 2\\
2&\SL_n\ \text{($n$ even, $n \ge 4$)}&\wedge^2(\k^n)&\binom{n}2&\Pf&n/2&\\
3&\E_6^\sc&\text{minuscule}&27&\text{see \cite[p.~358]{Jac:J}}&3& \\
4&\SL_n \times \SL_n&M_n&n^2&\det&n& \\ 
5&\SO_n\ (n \ge 3)&\k^n&n&&2&\left\{\parbox{0.5in}{$\ne 2$ if $n$ odd}\right. \\ \hline
6&\SL_2&\text{binary cubics}&4&\text{see Example \ref{binarycubics.f}}&4&\ne 2, 3\\
7&\SL_6&\wedge^3 (\k^6)&20&\text{see \eqref{SL6.f}}&\vdots&\vdots\\
8&\Sp_6&\wedge^3_0(\k^6)&14&\text{see \eqref{SL6.f}}&\vdots&\vdots\\
9&\Spin_{12}&\text{half-spin}&32&\text{see \cite[p.~1012]{Igusa}}&\vdots&\vdots\\
10&\E_7^\sc&\text{minuscule}&56&\text{see \cite{Ferr:strict}, \cite{Brown:E7}}&\vdots&\vdots \\
11&\SL_2 \times \SO_n\ (n \ge 4)&\k^2 \ot \k^n&2n&\text{see Example \ref{blackholes.f}}&4&\ne 2, 3 \\ \hline
\end{array}
\]
\caption{Some representations $V$ of groups $\Gt$ so that $f$ generates the ring of all polynomials on $V$ that are invariant under $\Gt$.  In these cases, we calculate the subgroup $\Stab_{\GL(V)}(f)$ of $\GL(V)$.} \label{1dim}
\end{table}

We maintain the assumptions of the previous sections, and from here on we assume furthermore that $V$ is irreducible and the ring $\k[V]^G$ of $G$-invariant polynomial functions on $V$ is generated by a non-constant homogeneous element that we denote by $f$.  
Since $V$ is an irreducible representation of $G$ and $f$ is not constant, the subspace consisting of $r \in V$ such that $f(r + v) = f(v)$ for all $v \in V$ must be zero.  It follows easily from this that every linear transformation preserving $f$ is \emph{invertible}, as noted in \cite[\S1]{Wa:inv}, hence 
the collection of linear transformations $\phi$ of $V$ that preserve $f$ is the group of $\k$-points of the closed sub-group-scheme $\Stab_{\GL(V)}(f)$ of $\GL(V)$.  We call it the \emph{preserver} of $f$; the classical linear preserver problem is to determine the $\k$-points of this group.

\begin{lem} \label{aclosed.id}
If $\k$ is algebraically closed, then 
\[
N_{\GL(V)}(G)^\circ(\k) \cap \Stab_{\GL(V)}(f)(\k) = G(\k) \cdot \mu_{\deg f}(\k).
\]
\end{lem}

\begin{proof}
By Corollary \ref{norm.gen}, every element of $N_{\GL(V)}(G)^\circ(\k)$ is a product $gz$ for some $g \in G(\k)$ and $z \in \kx$, and the claim is clear.
\end{proof}

In the notation of \eqref{Lt.hyp}, the equation 
\[
f(\rt(g)v) = \chi(g) f(v) \quad \text{for all $v \in V \ot \kalg$}
\]
defines a homomorphism $\chi \!: \Lt \ra \Gm$.  Corollary \ref{Lt.norm} immediately gives:
\begin{lem} \label{Lt.id}
Assuming \eqref{Lt.hyp}, the elements of $N_{\GL(V)}(G)^\circ(\k)$ that preserve $f$ are $\rt(g)$ for $g \in \Lt(\k)$ such that $\chi(g) = 1$. $\hfill\qed$
\end{lem}

As to the non-identity component of $N_{\GL(V)}(G)$, we have:

\begin{lem} \label{AutD.f}
There is a homomorphism $\phi \!: \Aut(\D, \la) \ra \Gm$ such that $f(\gamma(\pi) v) = \phi(\pi) f(v)$ for all $\pi \in \Aut(\D, \la)$ and $v \in V \ot E$ for every extension $E$ of $\k$.
\end{lem}

\begin{proof}
For a fixed $\pi \in \Aut(\D, \la)$, define $f_\pi \in \k[V]$ via $f_\pi(v) := f(\gamma(\pi) v)$.  As
\[
f_\pi(gv) = f(\gamma(\pi) g v) = f((i(\pi)(g)) \gamma(\pi) v) = f_\pi(v) \quad \text{for all $v \in V \ot E$,}
\]
and $f_\pi$ and $f$ are homogeneous of the same degree in $\k[V]^G$, we deduce that $f_\pi = \phi(\pi) f$ for some scalar $\phi(\pi) \in \k$.  As $f$ is nonzero on $V$, $f_\pi$ is also nonzero, hence $\phi(\pi) \in \kx$.
\end{proof}

\section{Lines 1--5 of Table \ref{1dim}} \label{RRS.sec}

For each of the polynomials $f$ appearing in lines 1--5 of Table \ref{1dim}, we will determine the linear transformations of $V$ that preserve $f$.  We prove the following, which is a formal version of an imprecise observation made in \cite[p.~840]{Marcus:AMMsurv}.

\begin{prop} \label{RRS.prop}
For the representations in lines 1--5 of Table \ref{1dim}, every linear transformation of $V$ that preserves $f$ belongs to $N_{\GL(V)}(G)(\k)$.
\end{prop}

Because we know so much about these representations, we can check this by hand in each case.  This is well known for line 4, is done for line 1 in \cite{Eaton}, and a similar argument using the generic minimal polynomial defined in \cite[Ch.~VI]{Jac:J} or \cite{G:dets} works for lines 2 and 3.  Alternatively, the proposition follows easily from the following:

\begin{lem}  \label{RRS} Suppose $\k$ is infinite.
For the representations in lines 1--5 of Table \ref{1dim}, $v \in V$, and an indeterminate $t$, we have: $v$ is minimal if and only if $\deg f(tv + v') \le 1$ for all $v' \in V$.
\end{lem}

 Again, this claim can be checked by hand in each case, as is done for line 4 in \cite[Lemma 3.2]{MarcusPurves}---note that the difficult direction in that paper is ``if".  (For line 5 this check is really trivial.)  We sketch a more uniform proof of this direction based on \cite{RRS}.

\begin{proof}[Proof of Lemma \ref{RRS}, ``if"] Whether or not $v$ is minimal is unchanged upon enlarging $\k$, and the same is true for the other condition (because $\k$ is infinite), so we may assume that $\k$ is algebraically closed.  The representation $G \ra \GL(V)$ is not only a representation as in \S\ref{LPP.1}, it is furthermore of the type considered in \cite{RRS} and in particular there is a sequence $u_1, \ldots, u_d$ of weight vectors in $V$ so that every element of $V$ is in the $G(\k)$-orbit of some $\sum_{i=1}^r c_i u_i$ for $c_i \in \kx$ (ibid., Th.~1.2(a)); an element is minimal if and only if it is in the orbit of $c_1 u_1$ for some $c_1 \in \kx$; and $f$ vanishes on $\sum_{i=1}^r c_i u_i$ if and only if $r < d$ (ibid., Prop.~2.15(b)). 

The number $d$ is calculated from root system data (ibid., p.~658), but in each case we see that it equals the degree of the invariant polynomial $f$ computed as described in \S\ref{LPP.1}.  We claim that the restriction of $f$ to the span of the $u_i$ is given by
\[
f( \sum_{i=1}^d c_i u_i ) = c \prod_{i=1}^d c_i \quad \text{for some $c \in \kx$.}
\]
Indeed, the normalizer of $T$ in $G$ permutes the $u_i$ arbitrarily (ibid., Th.~2.1), so the monomials appearing with a nonzero coefficient in the formula for the restriction of $f$ are stable under the obvious action by the symmetric group on $d$ letters.  The condition that $f(\sum_{i=1}^r c_i u_i)$ with $c_i \in \kx$ vanishes if and only if $r < d$ together with the degree of $f$ being $d$ implies the claimed formula.

Finally, if $v$ is non-minimal, then it is in the orbit of $\sum_{i=1}^r c_i u_i$ for some $r > 1$, and it is easy to produce a $v'$ so that $\deg f(tv + v') > 1$; this settles the ``if" direction.
\end{proof}

\begin{proof}[Proof of Proposition \ref{RRS.prop}]
Any linear transformation $\phi$ that preserves $f$ by definition also preserves $f$ over every extension of $\k$.  Hence, by Lemma \ref{RRS}, $\phi$ preserves minimal elements in $V \ot \kalg$, i.e., $\phi$ belongs to $\Stab_{\GL(V)}(\O)$, which equals $N_{\GL(V)}(G)$ by Theorem \ref{rk1}.
\end{proof}

\begin{rmk}
Lines 1, 2, and 4 of Table \ref{1dim} have in common that $V$ can be endowed with a bilinear multiplication that is ``strictly power associative" and so $V$ has a generic characteristic polynomial as mentioned above.  Write $E_r$ for the coefficient of the characteristic polynomial that is a homogeneous function on $V$ of degree $r$, so that $E_d = f$.  A uniform argument as in Lemma \ref{RRS} shows that the preserver in $\GL(V)$ of $E_r$ for $3 \le r < d$ is contained in $\Stab_{\GL(V)}(\O)$.  We omit the details, but the interested reader can find a precise description of the preserver of $E_r$ in \cite[Cor.~6.5]{Guralnick:LPP2} for symmetric matrices (line 1), \cite{MarcusWestwick} or \cite[Cor.~7.7]{Guralnick:LPP2} for alternating matrices (line 2), and \cite{MarcusPurves}, \cite{Beasley}, or \cite[Cor.~1]{Wa:basic} for square matrices (line 4).  (For line 3, one also has a generic characteristic polynomial, but $f$ is the only coefficient of degree $\ge 3$.)    
\end{rmk}

\begin{rmk}
Lines 1--5 of the table do not exhaust all the representations considered by \cite{RRS}.  The ones we have omitted lack a $G$-orbit of codimension 1 (ibid., Prop.~3.12) yet there is an open $G$-orbit in $\P(V)$, hence every $G$-invariant polynomial on $V$ is constant.
\end{rmk}

We can now determine the subgroup of $\GL(V)$ of elements that preserve $f$.  Our first result concerns symmetric matrices as in line 1 of the table.  Compare \cite[\S7.III]{Frobenius}, \cite[Th.~1]{Eaton}, \cite{Lim:symm},  \cite[Th.~6.7]{Wa:det}, or \cite[Cor.~6.3]{Guralnick:LPP2}.

\begin{cor}[symmetric matrices] \label{symm.f}
For $n \ge 2$ and $\k$ of characteristic $\ne 2$, every linear transformation $\phi$ of $\Sym_n(\k)$ that preserves the determinant is of the form \eqref{skew.1} where $r^n \det(P)^2 = 1$.
\end{cor}

\begin{proof}
Combine Proposition \ref{RRS.prop}, Lemma \ref{Lt.id}, and Corollary \ref{symm.rk1}.
\end{proof}

The next results concern alternating matrices as in line 2 of the table.  Compare \cite[Th.~3]{MarcusWestwick} (for $\k = \R$), \cite[Th.~5.7]{Wa:det}, or \cite[Cor.~7.4]{Guralnick:LPP2}.  

\begin{cor}[alternating $n$-by-$n$ matrices] \label{skew.f}
For even $n \ge 6$, every linear transformation $\phi$ of $\Skew_n(\k)$ that preserves the Pfaffian is of the form \eqref{skew.1} where $r^{n/2} \det(P) = 1$.  $\hfill\qed$
\end{cor}

\begin{cor}[alternating 4-by-4 matrices] \label{skew.f4}
Every linear transformation of $\Skew_4(\k)$ that preserves the Pfaffian is of the form \eqref{skew.1} or \eqref{skew.2} where $r^{n/2} \det(P) = 1$ 
\end{cor}

\begin{proof}
In view of Lemma \ref{AutD.f}, it suffices to pick some $X \in \Skew_4(\k)$ with $\Pf(X) \ne 0$ and verify that $\Pf(X^*) = \Pf(X)$ for $*$ as in \eqref{skew.star}.
\end{proof}

We now determine the linear transformations that preserve the determinant.  This is the case famously treated by Frobenius in \cite[\S7.I]{Frobenius} and Dieudonn\'e in \cite{Dieu:LPP}, and also in \cite[Th.~2]{MarcusMoyls:algebras} (for $\k = \C$) and \cite[Th.~4.2]{Wa:det}.

\begin{cor}[square matrices] \label{square.f}
Every linear transformation of $M_n(\k)$ that preserves the determinant is of the form \eqref{rect.1} or \eqref{rect.2} where $\det(AB) = 1$.$\hfill\qed$
\end{cor}

Every minuscule  representation $V$ of a group $G$ of type $E_6$ has a nonzero $G$-invariant cubic form $f$, and $G$-invariance uniquely determines $f$ up to multiplication by an element of $\kx$.  For the following result, compare \cite[7.3.2]{Sp:ex} (for $\car \k \ne 2, 3$) or \cite[5.4]{Asch:E6}.  
The analogous (and a priori coarser) result for Lie algebras is \cite[5.5.1]{Lurie}.

\begin{cor}[minuscule representation of $E_6$] \label{E6.f}
In the notation of the preceding paragraph, the preserver of $f$ in $\GL(V)$ is $G(\k)$.
\end{cor}

\begin{proof}
Since $\Aut(\D, \la) = 1$ and $\mu_3$ is in the center of $G$, Lemma \ref{aclosed.id} gives the claim.
\end{proof}

\subsection*{Commuting with the adjoint}
For the representations considered in this section, one has a notion of a ``classical adjoint" $\adj \!: V \ra V$, which is a polynomial map of degree $(\deg f) - 1$.  For lines 1, 2, and 4 of the table, the papers \cite{Sinkhorn} and \cite{ChanLimTan} compute the linear transformations on $V$ that commute with this map.  We can do the same for line 3, where $G$ is the simply connected split group of type $E_6$. The group $\Aut(\D)$ is $\Zm2$ and we write $\pi$ for the nonzero element; the subgroup of $G$ of elements fixed by $i(\pi)$ is a split group of type $F_4$ \cite[7.3]{CG} which we denote simply by $F_4$.  The center of $G$ is a copy of $\mu_3$.  We find:

\begin{cor} \label{e6.adj}
If $\car \k \ne 2, 3$, then the subgroup of $\GL(V)$ of elements commuting with the adjoint is $F_4(\k) \cdot \mu_3(\k)$.
\end{cor}

\begin{proof}
The minimal elements are precisely the nonzero $v \in V$ so that $\adj v = 0$ \cite[7.10]{CG}, so any element of $\GL(V)$ that preserves $\adj$ necessarily preserves minimal elements, hence belongs to $\Gm.G$.  Further, for any $g \in G$ and $c \in \Gm$, we have $\adj (cgv) = c^2 i(\pi)(g) \adj(v)$ \cite[7.9]{CG}, hence such a $cg$ commutes with $\adj$ if and only if $i(\pi)(g) = c^{-1} g$.  In particular, $c$ belongs to $G$ and so is a cube root of unity.  That is, the subgroup $H$ of $\GL(V)$ of elements commuting with the adjoint is contained in $G$, and the image of $H$ in the adjoint group $G/\mu_3$ is contained in the subgroup fixed by $i(\pi)$.

As $F_4 \times \mu_3$ is obviously contained in $H$, it suffices to show that its image in $G/\mu_3$ is the subgroup fixed by $i(\pi)$.  But this subgroup is connected reductive with Lie algebra of type $F_4$ \cite[7.3]{CG}, hence is the same as the image of $F_4$.  This proves the claim.
\end{proof}

\section{Lines 6--11 of Table \ref{1dim}} \label{FTS.sec}

The representations on lines 6--11 of Table \ref{1dim} are all of the form considered in \cite{Roe:extra}, \cite{Ferr:strict}, and \cite{Meyb:FT}.   In particular, the ring of $G$-invariant polynomials on $V$ is generated by a homogeneous polynomial $f$ of degree 4.    These representations appear, for example, when studying electromagnetic black hole charges in various supergravity theories, see \cite{BDFMR:small}. We suppose in this section that the characteristic of $\k$ is $\ne 2, 3$; the assumption that the characteristic is $\ne 2$ is so that we may apply the results of \cite{Roe:extra} and the assumption that the characteristic is $\ne 3$ is a convenience so that we may apply the results of \cite{Helenius}.  We will prove:

\begin{prop} \label{FTS.prop}
For the representations in lines 6--11 of Table \ref{1dim}, every element of $\GL(V)$ that preserves $f$ belongs to $N_{\GL(V)}(G)(\k)$.
\end{prop}

In view of our assumption on the characteristic, we are free to abuse notation and multilinearize $f$ to obtain a symmetric 4-linear form that we also denote by $f$.
Further, for each of these representations, there is a nondegenerate skew-symmetric bilinear form $\qform{\, , \,}$ on $V$ that is invariant under $G$.  This allows us to define a trilinear map $t \!: V \times V \times V \ra V$ implicitly by the equation:
\[
\qform{t(x_1, x_2, x_3), x_4} = f(x_1, x_2, x_3, x_4) \quad \text{for $x_1, x_2, x_3, x_4 \in V$.}
\]

For each $x \in V$, we define a symmetric bilinear form $b_x$ on $V$ via $b_x(v_1, v_2) = f(x, x, v_1, v_2)$.  We have:
\begin{lem} \label{radlem}
$x$ is a minimal element if and only if the dimension of the radical of $b_x$ is $(\dim V) - 1$.
\end{lem}

\begin{proof}
Line 6 is the representation from Example \ref{binarycubics.f}, for which we may check the claim of the lemma by hand.  So assume $G$ and $V$ come from one of the lines 7--11; we may apply results from \S3--4 of \cite{Helenius}. 

The radical of $b_x$ has codimension 1 if and only if it is the subspace $y^\perp$ of $V$ of vectors orthogonal (relative to $\qform{ \, , \,}$) to some nonzero $y \in V$.  That is, if and only if there is a nonzero $y \in V$ such that $f(x, x, y^\perp, z) = 0$ for all $z \in V$.  (For ``only if", one needs to know that $b_x$ is nonzero for $x$ nonzero, which is Lemma 14 in ibid.)  In turn, this is equivalent to: there is a nonzero $y \in V$ such that $t(x, x, z) \in \k y$ for every $z \in V$.  But by ibid., Propositions 18 and 20, that is the same as asking for $x$ to be minimal.
\end{proof}

\begin{proof}[Proof of Proposition \ref{FTS.prop}]
Suppose $\phi$ preserves $f$.  It defines an isometry between the bilinear forms $b_x$ and $b_{\phi(x)}$ for all $x \in V$.  Now apply Lemma \ref{radlem} to deduce that $\phi$ belongs to $\Stab_{\GL(V)}(\O)$, hence to $N_{\GL(V)}(G)$ by Theorem \ref{rk1}.
\end{proof}

We now determine the preserver of the discriminant of binary cubic forms as in line 6 or Example \ref{binarycubics.f}.
We omit the details in the proofs of this corollary and the following items because they are entirely similar to earlier proofs.

\begin{cor}[binary cubics] \label{cubics}
Every linear transformation on the vector space of cubic forms $\k^2 \ra \k$ that preserves the discriminant is of the form $q \mapsto c q \circ g$ for some $c \in \kx$ and $g \in \GL_2(\k)$ such that $c^4 (\det g)^6 = 1$.$\hfill\qed$
\end{cor}

Line 7 of the table concerns an $\SL_6$-invariant quartic form $f$ on $\wedge^3 \k^6$, for which a (complicated-looking) formula is given in \cite[p.~83]{SK}; we now give an alternative presentation.  Write $\k^6$ as a direct sum $V_2 \oplus V_4$ where $V_d$ has dimension $d$.  There is a natural inclusion $w: V_2 \ot (\wedge^2 V_4) \ra \wedge^3 \k^6$ given by $(c, x) \mapsto c \wedge x$.  Amongst the line of invariant quartic forms on $\wedge^3 \k^6$, there is an element $f$ so that, with respect to a fixed basis $a, b$ of $V_2$, we have:
\begin{equation} \label{SL6.f}
f(w(a \ot x + b \ot y)) = \qform{x, y}^2 - 4 \Pf(x) \Pf(y)
\end{equation}
where $\qform{x , y}$ denotes the coefficient of $t$ in $\Pf(x + ty)$, i.e., the polarization of the $\SL(V_4)$-invariant quadratic form $\Pf$.  (To check the claim \eqref{SL6.f}, one can either use the formula for $f$ in \cite{SK} or one can observe that $fw$ is a nonzero quartic form on $V_2 \ot (\wedge^2 V_4)$ that is invariant under $\SL(V_2) \times \SL(V_4)$, that there is a unique line of such forms if $\k = \C$, and that the right side of \eqref{SL6.f} gives another such form.)  As every $\SL_6(\k)$-orbit in $\wedge^3 \k^6$ meets the image of $w$ by \cite[Lemma 2.2]{Revoy:6}, equation \eqref{SL6.f} is enough to specify $f$ on $\wedge^3 \k^6$.

\begin{cor} \label{SL6}
Every linear transformation of $\wedge^3 (\k^6)$ that preserves the invariant quartic form is of the form 
\begin{equation} \label{SL6.1}
v_1 \wedge v_2 \wedge v_3 \mapsto c(gv_1 \wedge gv_2 \wedge gv_3) \ \text{for $c \in \kx$, $g \in \GL_6$ with $c^4 (\det g)^2 = 1$}
\end{equation}
or the composition of a Hodge star operator with a transformation as in \eqref{SL6.1}.$\hfill\qed$
\end{cor}

Regarding line 8 of the table, recall that $\Sp_6$ is defined as the subgroup of $\GL_6$ leaving a particular nondegenerate skew-symmetric bilinear form $b$ invariant on its (tautological) representation $\k^6$.  We write $\wedge^3_0(\k^6)$ for the kernel of the contraction map $\wedge^3 (\k^6) \ra \k^6$, cf.~\cite[\S17.1]{FH}.  The restriction of the $\SL_6$-invariant quartic form on $\wedge^3 \k^6$ to $\wedge^3_0 (\k^6)$ gives an $\Sp_6$-invariant quartic form.

We define $\GSp_6$ to be the subgroup of $\GL_6$ of transformations that scale the bilinear form $b$ by a factor in $\kx$; it is isomorphic to $(\Sp_6 \times \Gm)/\mu_2$.  

\begin{cor} \label{Sp6}
Every linear transformation of the space $\wedge^3_0(\k^6)$ that preserves the invariant quartic is of the form
\[
v_1 \wedge v_2 \wedge v_3 \mapsto c(gv_1 \wedge gv_2 \wedge gv_3) \quad \text{for some $c \in \kx$ and $g \in \GSp_6(\k)$}
\]
with $c^4 (\det g)^2 = 1$.$\hfill\qed$
\end{cor}

For the representations on lines 9 and 10, we will prove a result under the assumption that $\k$ contains a square root of $-1$.  Alternatively, we could eliminate this hypothesis at the cost of defining a reductive envelope $\Lt$ of $G$ as we defined $\GSp_6$ for $\Sp_6$ above, i.e., as in \eqref{Lt.hyp}.

We write $\HSpin_{12}$ for the image $G$ of $\Spin_{12}$ under a half-spin representation.
\begin{cor} \label{HSpin12} 
Suppose $\k$ contains a square root of $-1$.  Then the subgroup of $\GL_{32}(\k)$ of transformations that preserve the $\HSpin_{12}$-invariant quartic form is $\HSpin_{12}(\k) \cdot \mu_4(\k)$.$\hfill\qed$
\end{cor}

For the next result, compare  \cite[Cor.~2.6(i)]{Sp:E7} or \cite[\S10]{Helenius}.  Those proofs are based on versions of Corollary \ref{E6.f} for $\E_6$, but our proof does not refer to $\E_6$.

\begin{cor} \label{E7}
Suppose $\k$ contains a square root of $-1$.  Then the subgroup of $\GL_{56}(\k)$ of transformations that preserve the $\E^\sc_7$-invariant quartic form is $\E^\sc_7(\k) \cdot \mu_4(\k)$.$\hfill\qed$
\end{cor}

As for line 11, we consider first the case $n = 4$.  As the automorphism group of the Dynkin diagram of $\SL_2 \times \SL_2 \times \SL_2$ is the symmetric group $\mathcal{S}_3$, we have the following result (compare \cite[\S11]{Helenius}):
\begin{cor} \label{hyperdet}
Every linear transformation of $\k^2 \ot \k^2 \ot \k^2$ that preserves the hyperdeterminant is of the form
\[
v_1 \ot v_2 \ot v_3 \mapsto g_1 v_1 \ot g_2 v_2 \ot g_3 v_3 \quad \text{for $g_1, g_2, g_3 \in \GL_2(\k)$}
\]
 such that $\det(g_1 g_2 g_3) = \pm 1$, or is the composition of such a map with a permutation 
\[
v_1 \ot v_2 \ot v_3 \mapsto v_{\s(1)} \ot v_{\s(2)} \ot v_{\s(3)} \quad \text{for $\s \in \mathcal{S}_3$.}  
\]
\end{cor}

\begin{proof}
In the notation of \eqref{Lt.hyp}, one takes $\Lt$ to be a product of 3 copies of $\GL_2$ with the obvious $\rt$; the kernel of $\rt$ is isomorphic to $\Gm \times \Gm$.
In view of Proposition \ref{RRS.prop} and \S\ref{LPP.soln}, it suffices to check that the permutations preserve the hyperdeterminant, which is clear from the explicit formula for the hyperdeterminant from, e.g., \cite{MiyakeWadati}.
\end{proof}

For the representations on line 11 with $n \ge 5$, we define $\GO_n$ to be the algebraic group with $R$-points the matrices $g \in \GL_n(R)$ such that $g^t S g = \mu(g) S$ for some $\mu(g) \in R^\times$ (for every $\k$-algebra $R$); it is a reductive envelope of $\mathrm{O}_n$.

\begin{cor} \label{blackholes}
For $n \ge 5$, every linear transformation of $M_{2n}(\k)$ that preserves the degree $4$ function from Example \ref{blackholes.f} is of the form
\[
X \mapsto g_1 X g_2^t \quad \text{for $g_1 \in \GL_2(\k)$, $g_2 \in \GO_n(\k)$ with $\det(g_1)\mu(g_2) = \pm 1$.} 
\]
\end{cor}

\begin{proof}[Sketch of proof]
Note that $\Aut(\D, \la)$ is naturally identified with the component group of $\GO_n$.
\end{proof}

As a concrete illustration of the remarks in \S\ref{nonsplit.sec}, we note that Corollary \ref{blackholes} and Example \ref{blackholes.f} go through with no change if we replace the split groups $\SO_n$ and $\mathrm{O}_n$ with the special orthogonal and orthogonal groups of any nondegenerate symmetric bilinear form, i.e., where the matrix $S$ in Example \ref{blackholes.f} is any symmetric invertible matrix.  In this way, the corollary gives the stabilizer of $f$ also in the case where $\k = \R$ and $\SO_n$ is replaced by a real group  $\SO(2,n-2)$ or $\SO(6,n-6)$; this situation appears in the study of electromagnetic black hole charges in $\mathcal{N} = 2$ or $4$ supergravity, see e.g.~\cite{BDFMR:small}.

\section{Some representations omitted from Table \ref{1dim}} \label{ignore}

We have not yet discussed all pairs $(G, V)$ where $G$ is absolutely almost simple, $V$ is an irreducible representation of $G$ and $\k[V]^G$ is generated by a homogeneous polynomial $f$.  For $\k = \C$, all such pairs are listed in the the table on pages 260--262 of \cite{PoV}, and we now discuss each of the cases that we have thus far omitted.  

\smallskip
Consider one of the groups $\HSpin_n$ for $n = 7, 9$ with their natural representations or $G_2$ with its 7-dimensional representation.  In these cases, $f$ has degree 2, i.e., is a quadratic form, so its linear preserver is the orthogonal group $\mathrm{O}(f)$.  Note that our Theorem \ref{rk1} does not apply to these groups because they correspond to exclusions \eqref{Dem.HSpin} and \eqref{Dem.G2} in Definition \ref{dem.def}.

The natural 32-dimensional representation of the group $\HSpin_{11}$ factors through the natural representation of $\HSpin_{12}$.  The ring $\k[V]^{\HSpin_{12}}$ is also 1-dimensional (as can be seen by the reasoning in \S\ref{LPP.1}, where $H$ has type $E_7$) with generator $f$ of degree 4, so clearly the $f$ stabilized by $\HSpin_{11}$ is the same as for $\HSpin_{12}$ and so the linear preserver of this $f$ is $\HSpin_{12}.\mu_4$ as in Corollary \ref{HSpin12}.  (For generalizations of this sort of example, see \cite{Solomon:irredsame}.)

\smallskip
The only remaining pairs $(G, V)$ are $(\SL_7, \wedge^3 (\k^7))$, $(\SL_8, \wedge^3(\k^8))$, and $\HSpin_{14}$ with its natural $120$-dimensional representation.  The first representation is noteworthy, because the stabilizer in $\SL_7$ of any element $v$ such that $f(v) \ne 0$ is a group of type $G_2$, see for example \cite[p.~65]{Engel}, \cite{Asch:G2},  
or \cite{CohenHel}.  The orbits in the last representation have been studied over $\C$ in \cite{Popov:14}, and the fact that $\dim V/G = 1$ has been applied to the theory of quadratic forms in \cite{Rost:14.1} and \cite{Rost:14.2}, see also \cite{G:lens}. All three of these representations are irreducible and are not stable under an outer automorphism of $G$, so applying Theorem \ref{rk1}, we find without doing any work that $\Stab_{\GL(V)}(\O)$ is $\Gm.G$.  As to the preserver $\Stab_{\GL(V)}(f)$ in these cases, we omit serious investigation.  However, for $\k = \C$, one can observe that the identity component $G'$ of $\Stab_{\GL(V)}(f)$ is reductive (because $V$ is an irreducible representation), hence is semisimple (because the center must consist of scalar matrices).  It follows from the classification of semisimple groups $G'$ such that $\C[V]^{G'}$ is generated by a single polynomial that $G' = G$.  In particular, $G$ is normal in $\Stab_{\GL(V)}(f)$.  As $\Aut(\D, \la) = 1$, it follows that $\Stab_{\GL(V)}(f)$ is contained in $G.\Gm$, i.e., $\Stab_{\GL(V)}(f)$ is $G.\mu_d$, where $d$ is the degree of $f$ (equal to 7, 16, or 8 respectively).

\section{An alternative formulation of the linear preserver problem} \label{LPP.2}

Inspecting the LPP solutions by Frobenius (1897) and Dieudonn\'e (1949) where $V$ is the $n$-by-$n$ matrices and $f$ is the determinant, one sees that Frobenius determines the preserver of $\det$ whereas Dieudonn\'e determines the linear transformations on $V$ that preserve the set of singular matrices.  
So far, we have been solving Frobenius' version of the problem, but in fact we have also solved Dieudonn\'e's version:

\begin{prop} \label{dieu.prop}
For each of the representations in Table \ref{1dim}, the collection of linear transformations preserving the projective variety $f = 0$ is $N_{\GL(V)}(G)$.
\end{prop}

See e.g.\ \cite{Schwarz:lmpf} for general results on the relationship between the two versions.

\begin{proof}
Put $S$ for the sub-group-scheme of $\GL(V)$ preserving the projective variety $f = 0$.  
Given any $s \in S(\kalg)$, ${^s}f$ is in the ideal generated by $f$ and has the same degree as $f$, hence ${^s}f = cf$ for some $c \in \kx$ and $c^{-1/\deg f} s$ preserves $f$.  Propositions \ref{RRS.prop} and \ref{FTS.prop} give that $s$ belongs to $N_{\GL(V)}(G)$.

Conversely, for $n \in \GL(V)$ normalizing $G$, Corollary \ref{norm.gen} and Lemma \ref{AutD.f} show that $^n f$ is a scalar multiple of $f$.
\end{proof}

Note that the proposition  indeed solves Dieudonn\'e's version of the linear preserver problem for the representations in Table \ref{1dim}, because
we calculated the group $N_{\GL(V)}(G)$ in \S\ref{norm.sec}.

\bibliographystyle{amsalpha}
\bibliography{skip_master}

\end{document}